\documentclass[11pt]{amsart}
\usepackage{amssymb}
\usepackage{amsmath}
\usepackage{amsfonts}
\usepackage{graphicx}

\usepackage{hyperref}
    \usepackage{aeguill}
    \usepackage{type1cm}

\theoremstyle{plain}
\newtheorem{thm}{Theorem}[section]

\newtheorem{claim}[thm]{Claim}

\newtheorem{lemma}[thm]{Lemma}

\newtheorem{remark}[thm]{Remark}

\newtheorem{theorem}[thm]{Theorem}
\newtheorem{ex}[thm]{Example}
\newtheorem{proposition}[thm]{Proposition}

\newtheorem{defn}[thm]{Definition}
\numberwithin{equation}{section}
\newcommand{\N}{\mathbb{N}}

\newcommand{\R}{\mathbb{R}}

\newtheorem{cor}{Corollary}[section]

\usepackage[usenames,dvipsnames]{color}

\usepackage[dvipsnames]{xcolor}

\begin{document}

\title[Nonsmooth Morse-Sard theorems]{Nonsmooth Morse-Sard theorems}

\author{D. Azagra}
\address{ICMAT (CSIC-UAM-UC3-UCM), Departamento de An{\'a}lisis Matem{\'a}tico,
Facultad Ciencias Matem{\'a}ticas, Universidad Complutense, 28040, Madrid, Spain}
\email{azagra@mat.ucm.es}

\author{J. Ferrera}
\address{IMI, Departamento de An{\'a}lisis Matem{\'a}tico,
Facultad Ciencias Matem{\'a}ticas, Universidad Complutense, 28040, Madrid, Spain}
\email{ferrera@mat.ucm.es}

\author{J. G\'omez-Gil}
\address{Departamento de An{\'a}lisis Matem{\'a}tico,
Facultad Ciencias Matem{\'a}ticas, Universidad Complutense, 28040, Madrid, Spain}
\email{gomezgil@mat.ucm.es}

\date{\today\ (\the\time)}


\thanks{The authors were partially supported by Grant MTM2012-34341 of Ministerio de Econom\'{i}a y Competitividad. D. Azagra was also partially  supported by ICMAT Severo Ochoa project SEV-2015-0554.}

\subjclass[2010]{}

\begin{abstract}
We prove that every function $f:\R^n\to \R$  satisfies that the image of the set of critical points at which the function $f$ has Taylor expansions of order $n-1$ and non-empty subdifferentials of order $n$ is a Lebesgue-null set. As a by-product of our proof, for the proximal subdifferential $\partial_{P}$, we see that for every lower semicontinuous function $f:\R^2\to\R$ the set $f(\{x\in\R^2 : 0\in\partial_{P}f(x)\})$ is $\mathcal{L}^{1}$-null. 
\end{abstract}

\maketitle

\section{Introduction and main results}

The main purpose of this paper is to provide nonsmooth versions of the Morse-Sard Theorem for real-valued functions defined on $\R^n$. Recall that the Morse-Sard theorem \cite{Morse, Sard} states that if $f:\R^n\to\R^m$ is of class $C^k$, where $k=n-m+1$, then the set of critical values of $f$ has measure zero in $\R^m$. A famous example of Whitney's \cite{Whitney} shows that this classical result is sharp within the classes of functions $C^j$. Nevertheless several generalizations of the Morse-Sard theorem for other classes of functions have appeared in the literature; see 
\cite{Alberti, Bates, BoHaSt, BourKoKris1, BourKoKris2, DePascale, Dubo, Figalli, KorobkovKristensen, HajlaszZimmerman, Norton, NortonZMS, Landis, PavZaj, Putten, Rifford, Yomdin} and the references therein. We cannot state all of the very interesting results of the rich litterature concerning this topic; instead, because of its pointwise character which is closely related to our results, let us only mention that Bates proved in \cite{Bates} that if $f\in C^{k-1,1}(\R^n,\R^m)$ (i.e., if $f\in C^{k-1}$ and $D^{k-1}f$ is Lipschitz) then the conclusion of the Morse-Sard theorem still holds true. In \cite{AFG} we gave an abstract version of the Morse-Sard theorem which allows us to recover a previous result of De Pascale's for the class of Sobolev functions \cite{DePascale}, as well as a refinement of Bates's theorem which only requires $f$ to be $k-1$ times continuously differentiable and to satisfy a Stepanoff condition of order $k$, namely that $$\limsup_{h\to 0}\frac{|f(x+h)-f(x)-Df(x)(h) - ... - \frac{1}{(k-1)!} D^{k-1}f(x)(h^{k-1})|}{|h|^k}<\infty$$ for every $x\in\R^n$. As a referee of the present paper pointed out, this result can also be easily proved, and even generalized, by using some ideas of the proof of \cite[Theorem 1]{LiuTai}; see the Appendix below.

In the present paper we will look at the case $m=1$ more closely, and we will study the question as to what extent one-sided Taylor expansions (that is, viscosity subdifferentials) of order $n$ are sufficient to ensure that a given function $f:\R^n\to\R$ has the Morse-Sard property. The results that we obtain generalize many of the previous versions of Morse-Sard Theorem and do not require
that the function $f$ be $C^{n-1}$ smooth (nor even two times differentiable). They are meant to complement the nonsmooth versions of the Morse-Sard theorem for subanalytic functions and for continuous selections of compactly indexed countable families of $C^{n}$ functions on $\R^n$ that were established in \cite{BolteDaniilidisLewis, BarbetDambrineDaniilidis}.

For an integer $n\geq 2$, we will say that
a function $f:\R^N\to\R^m$ has a Taylor expansion of order $n-1$ at $x$ provided there exist
$k$-homogeneous polynomials $P_x^k$, $k=1,\dots ,n-1$, such that
$$
\lim_{h\to 0}\frac{f(x+h)-f(x)-P_x^1(h)-P_x^2(h)-\dots -P_x^{n-1}(h)}{|h|^{n-1}}=0.
$$
If there exist such polynomials then they are unique. Also note that if a function $f$ has Taylor expansion of order $n-1$ at a point $x$, then it is
differentiable at $x$ and the differential $Df(x)$ equals the linear function $P_x^1$; however $D^{j}f(x)$ does not necessarily exist for $j\geq 2$. On the other hand, if $f$ is $n-1$ times differentiable at $x$ then $f$ has a Taylor expansion of order $n-1$ at $x$, and $P_x^k=\frac{1}{k!}D^kf(x)$ for every $k=1,\dots ,n-1$.
For more information on Taylor expansions and its relation with approximate differentiability and Lusin properties of higher order, see \cite{LiuTai, LinLiu}.

Let us now explain what we mean by a subdifferential of order $n$.
Probably, the most natural way to define a subdifferential $\tilde{\partial}^nf(x_0)$  of order $n$ of a lower semicontinuous function $f:\R^N\to\R$ at a point $x_0$ 
is as the set of $n$-tuples  $(P_1,\dots ,P_n)\in \mathcal{P} (^1\R^N)\times \dots \mathcal{P} (^n\R^N)$ such that
$$
\liminf_{x\to x_0}\frac{f(x)-f(x_0)-P_1(x-x_0)-\dots -P_n(x-x_0)}{|x-x_0|^n}\geq 0.
$$
Here $\mathcal{P} (^k\R^N)$ denotes the space of $k$-homogeneous polynomials on $\R^N$, which is endowed with the norm
$$
\|P\|=\sup_{|v|=1}|P(v)|.
$$
In the case $n=2$ this definition agrees with the standard viscosity subdifferential of order $2$; see \cite{CIL} and the references therein.
It is easy to see that if $(P_1,\dots ,P_n)\in \tilde{\partial}^nf(x_0)$ then  $(P_1,\dots ,P_{n-1})\in \tilde{\partial}^{n-1}f(x_0)$.
It is also clear that if the polynomial $\varphi (x)=f(x_0)+P_1(x-x_0)+\dots P_{n-1}(x-x_0)$ satisfies $\varphi \leq f$ on a neighbourhood
of $x_0$, then $(P_1,\dots ,P_{n-1})\in \tilde{\partial}^{n-1}f(x_0)$.
For $n$ odd, the converse is partially true, in the following sense:
if $(P_1,\dots ,P_{n-1})\in \tilde{\partial}^{n-1}f(x_0)$ and $\varepsilon >0$, then the polynomial
 $\varphi (x)=f(x_0)+P_1(x-x_0)+\dots P_{n-1}(x-x_0)-\varepsilon |x-x_0|^{n-1}$ is less than or equal to
$f$ on a neighbourhood of $x_0$ (this is not necessarily true if $n$ is even). Hence, we have the following.

\begin{proposition}\label{nonemty subd of order 2 implies nonempty subd for every n}
If $(\zeta ,P)\in \tilde{\partial} ^2f(x_0)$ and $P_{\varepsilon}(h)=P(h)-\varepsilon |h|^2$
then $(\zeta ,P_{\varepsilon},0,\dots ,0)\in \tilde{\partial} ^nf(x_0)$ for every $\varepsilon >0$ and every $n\geq 3$.
\end{proposition}
\noindent However, this does not imply that $(\zeta ,P,0)\in \tilde{\partial} ^3f(x_0)$. In particular, we see that the subdifferential $\tilde{\partial} ^n f(x_0)$
as a subset of
$\mathcal{P} (^1\R^N)\times \dots \mathcal{P} (^n\R^N)$, is not necessarily closed for $n\geq 2$ (althought it is always closed for $n=1$).
Technical problems arise from this fact when one tries to extend the tools of nonsmooth analysis to higher order subdifferentials. In order to overcome these problems one could try to introduce limiting subdifferentials of higher order, but this would not lead us anywhere as far as nonsmooth Morse-Sard theorems are concerned; see Example \ref{MS fails for the limiting sbd} below. Another important theoretical disadvantage of this
subdifferential is the fact that there is no gap between subdifferentiability of the orders 2 and 3 (nor between subdifferentiability of order
$n-1$ and $n$ more generally). Namely, if $n\geq 3$ then there are no functions with nonempty subdifferential of order $n-1$ and with empty subdifferential of order $n$. 

For these reasons we introduce in this paper a slightly different subdifferential which will suit our investigation concerning nonsmooth generalizations of the Morse-Sard theorem.

\begin{defn}
{\em Let $f:\R^N\to \R$, $x_0\in \R^N$. If $f$ has a Taylor expansion of order $n-1$ at $x_0$,
we define $\partial ^nf(x_0)$ as the set of  $Q\in \mathcal{P} (^n\R^N)$ such that
$$
\liminf_{h\to 0}\frac{f(x_0+h)-f(x_0)-Df(x_0)(h)-P_{x_0}^2(h)-\dots -P_{x_0}^{n-1}(h)-Q(h)}{|h|^n}\geq 0.
$$
If either $f$ does not have a Taylor expansion of order $n-1$ at $x_0$, or there does not exist any $Q$ with such property, then we deem $\partial ^nf(x_0)$ to be empty.}
\end{defn}

In order to illustrate this definition let us have a look at two examples of functions, both of which are of class $C^2$, but the first one has a big  subdifferential of order $3$, while the second one has empty subdifferential of order $3$. For the function $f_1:\R \to \R$ defined by
$$
f_1(x)=
\left\{
  \begin{array}{ll}
    -x^2 & \hbox{if } x\leq 0; \\
    -x^2+x^3 & \hbox{otherwise},
  \end{array}
\right.
$$
we have $Df_1(0)\equiv 0$, $D^2f_1(0)(h)=-2h^2$, and $T\in \partial ^3f_1(0)(h)$ if and only if $T(h)=ah^3$
with $a\in [0,1]$.
However, for the function $f_2(x)=-|x|^3$, we have
$Df_2(0)\equiv 0$, $D^2f_2(0)\equiv 0$ and $\partial ^3f_2(0)=\emptyset$.

The subdifferential $\partial$ that we have just introduced is smaller and behaves better than $\tilde{\partial}$ does. For instance:

\begin{proposition}
The set $\partial ^nf(x_0)$ is convex and closed.
\end{proposition}
\noindent The proof is straightforward. 

However, $\partial ^nf(x_0)$ may still have a complicated structure. The following example shows that if $n\geq 2$ we may have to face the following situation (which already in the case $n=2$ prevents more straightforward strategies than the one we use from working out):
$$
Df(x_0)= 0, \ \ \ \sup_{|v|=1}P(v)<0 \ \hbox{for every} \ P\in \partial ^2f(x_0),
$$
and yet
$$
\sup \{ P(v): |v|=1,  P\in \partial ^2f(x_0)\} =0.
$$

\begin{ex}
{\em Define $P_n(x,y)=-(nx^2+\frac{1}{n}y^2)$. The function $f:\R^2\to \R$ defined by
$$
f(x,y)=\sup_nP_n(x,y)
$$
satisfies: $Df(\bar{0})=\bar{0}$,  $P_n\in \partial ^2f(\bar{0})$ for every $n$, and
$\sup_{|v|=1}P(v)<0$ for every $P\in \partial ^2f(\bar{0})$.}
\end{ex}
\begin{proof} First we observe that $f$ is a 2-homogeneous function and
  \begin{equation}\label{eq:1}
    -|(x,y)|^2 \leq f(x,y)\leq \min\{-x^2,-2|xy|\}.
  \end{equation}
From the first inequality it follows that $Df(\bar{0})=\bar{0}$. On the other hand $P_n\leq f$, hence $P_n\in \partial ^2f(\bar{0})$.

As $f$ is a 2-homogeneous function, $P\in   \partial ^2f(\bar{0})$ if, and
only if $P(v)\leq f(v)$ for all $v\in \mathbb{R}^2$. In this case, by
\eqref{eq:1}, if $P(x,y)=0$ then $x=0$. This implies that if 
$P(x,y)=-a\langle (x,y),e\rangle ^2$, $a>0$ and $e=(e_1,e_2)\in
\mathbb{S}^1$, then $P\notin \partial ^2f(\bar{0})$.
Indeed, if $P\in \partial ^2f(\bar{0})$, as $P(-e_2,e_1)=0$ then $e_2=0$ and therefore
\begin{equation*}
  P(1,y)=-a\leq f(1,y)\leq -2|y|
\end{equation*}
for all $y$, which  is not possible.
This implies that $P<0$ for every $P\in \partial ^2f(\bar{0})$.
\end{proof}

If $f$ has a Taylor expansion of order $n-1$ at $x_0$, we may also define the superdifferential of order $n$ as the set $\partial _{+}^nf(x_0)$ of all $n$-homogeneous polynomials $Q$ satisfying
$$
\limsup_{h\to 0}\frac{f(x_0+h)-f(x_0)-Df(x_0)(h)-P_x^2(h)-\dots -P_x^{n-1}(h)-Q(h)}{|h|^n}\leq 0.
$$
It is then clear that if $Q\in \partial _{+}^nf(x_0)\cap \partial^n f(x_0)$ then $f$ has Taylor expansion of order $n$ and $Q$ is the unique polynomial with this property. We also have the following result, whose proof is straightforward.

\begin{proposition}
For $n$ an odd integer, a function $f$ has a Taylor expansion of order $n$ at a point $x$ if and only if it has Taylor expansion of order $n-1$
at $x$ and $\partial^nf(x)\neq\emptyset \neq \partial _{+}^nf(x)$. In this case we also have
then $\partial^n f(x) =\partial _{+}^nf(x)=\{ P^{n}_{x}\}$, where $P^{n}_{x}$ is the $n$-homogeneous part of the Taylor expansion of order $n$ of $f$.
\end{proposition}

Other interesting properties of this subdifferential can of course be established. However, our motivation to introduce the higher order subdifferential $\partial^n$ is the fact that with this tool we will be able to obtain the following very general version of the Morse-Sard theorem for real-valued functions, which is the main result of this paper.

\begin{thm}\label{main theorem}
Let $f:\R^n \to \R$ be a function, $n\geq 2$, and let $C_f$ be the set of points  $x\in\R^n$ such that $\partial^{n}f(x)\neq\emptyset$ and $Df(x)=0$. 
Then $\mathcal{L}^1 \bigl( f(C_f)\bigr) =0$.

\noindent The same statement holds true if $\R^n$ is replaced with an open subset of $\R^n$.
\end{thm}
\noindent Here, as in the rest of the paper, $\mathcal{L}^{N}$ denotes Lebesgue's outer measure in $\R^N$. Notice that if $\partial^n f(x)\neq\emptyset$ and $n\geq 2$ then, according to our definition of $\partial^n f(x)$, $f$ has a Taylor expansion of order $n-1$ at $x$, and in particular $Df(x)$ exists.

Our method of proof will also allow us to establish sharper versions of Theorem \ref{main theorem} in the special cases $n=2, 3$. A similar result for the case $n=1$ is easy and probably known, but nonetheless we include a proof for the reader's convenience. Perhaps the most interesting one is that of the case $n=2$, for which we obtain the following Morse-Sard theorem for the proximal subdifferential. Recall that for a lower semicontinuous function $f:\R^d\to (-\infty, \infty]$ the proximal subdifferential of $f$ (at a point $x$ where $f(x)<\infty$) is denoted by $\partial_{P}f(x)$ and defined as the set of all $\zeta\in\R^d$ for which there exist $\sigma, \eta>0$ such that
$$
f\left( y \right) \geq f\left( x \right) + \left\langle {\zeta ,y
- x} \right\rangle  - \sigma |y - x|^2 
$$
for all $y \in B\left( {x,\eta } \right)$. The set $\partial_{P}f(x)$ coincides with $\{\zeta\in\R^d : \zeta=\nabla\varphi(x), \varphi\in C^2(\R^d), f-\varphi$ attains a minimum at $x\}$.
\begin{thm}\label{theorem for the proximal subdifferential in the plane}
Let $f:\R^2\to \R$ be a lower semicontinuous function. Then
$$
\mathcal{L}^1\left(f\left(\{ x\in \R^2: 0\in \partial _Pf(x)\}\right)\right)=0.
$$
\end{thm}
Note that the above result generalizes \cite[Theorem 8]{Rifford}. On the other hand, the following example shows that there are functions which are not in the class $BV_{2}(\R^2)$ (and therefore cannot be concluded to have the Morse-Sard property by using the Bourgain-Korobkov-Kristensen Theorem of \cite{BourKoKris1, BourKoKris2}) but do satisfy the mild assumptions of Theorems \ref{theorem for the proximal subdifferential in the plane} and \ref{main theorem}. Recall that the class $BV_{n}(\R^n)$ is defined as the set of all integrable functions whose distributional derivatives of order $n$ are finite Radon measures; see \cite{Dorronsoro} for information about differentiability properties of these functions.
\begin{ex}
{\em Let $C\subset [0,1]$ be a Cantor-like set of positive measure.
Construct a continuous function $g:[0,1]\to \R$ as follows. Set $g(x)=0$ for every $x\in C$ and, for
each of the $2^{n-1}$ intervals $I_n^j$ of length $l_n$ that are removed from an interval $I^{k}_{n-1}$ at step $n$ in the construction of $C$,
consider a subinterval $J_n^j$ of length $\frac{l_n}{3}$ centered at the same point as $I_n^j$. Define $g$ on $I_n^j$ as a differentiable function which is not of bounded variation and such that $0\leq g(x)\leq l_n^{\frac{3}{2}}$ and $g(x)=0$ for every
$x\in I_n^j\setminus J_n^j$.
The function $F:(0,1)^2\to \R$ defined by $F(x,y)=f(x)+f(y)$ with $f(x)=\int_0^xg(t)dt$ satisfies $C\times C\subset C_F$ and has a Taylor expansion of
order two at every point, but it does not have a $BV$ derivative (and in particular $g$ is not $C^{1,1}_{\textrm{loc}}$ either). However, $F$ satisfies the hypotheses of Theorem \ref{main theorem}, and consequently has the Morse Sard property. }
\end{ex}

For the case $n=3$ we have the following.

\begin{thm}\label{Theorem for n=3}
Let $f:\R^3\to \R$ be a lower semicontinuous function. Then  $\mathcal{L}^{1} \bigl( f(C_f)\bigr) =0$,
where $C_f$ is defined as the set of $x\in\R^3$ for which $Df(x)=0$ and
$\partial ^3f(x)$ is nonempty.
\end{thm}

\begin{remark}
{\em It is impossible to have Theorem \ref{main theorem} or Theorem \ref{Theorem for n=3} if we replace $\partial ^nf(x)$ with
$\tilde{\partial}^nf(x)$. Indeed, as is well known from Whitney's \cite{Whitney} and others' examples, there exist $C^2$ functions on $\R^3$  that fail to have the
Morse-Sard Property;  however, as we observed in Proposition \ref{nonemty subd of order 2 implies nonempty subd for every n} we have $\tilde{\partial}^3f(x)\neq \emptyset$ for every $f\in C^{2}(\R^3)$ and every $x\in\R^3$.}
\end{remark}

In the case $n=1$ we also have the following special result.

\begin{proposition}\label{Theorem for n=1}
Let $f:\R \to \R$ be a lower semicontinuous function, and denote
$
C_f=\{ x\in \R: 0\in \partial f(x)\}.
$
Then $\mathcal{L}^1 \bigl( f(C_f)\bigr) =0$.
\end{proposition}
The following example taken from \cite{ClarkeEtAl} shows that the preceding proposition
is no longer true if we replace $\partial f$ with the $\partial_{L}f$, the limiting subdifferential of $f$ (see \cite{ClarkeEtAl, CIL, Ferrera} and the references therein for definitions and background on various subdifferentials).

\begin{ex}\label{MS fails for the limiting sbd}
{\em Let $C\subset [0,1]$ be a measurable set such that $0<\mathcal{L}^1 (C\cap I)<1$ for every interval $I$. Let
$$
f(x)=\int_0^x\chi _C(t)dt.
$$
It is easy to see that $0\in \partial _Lf(x)$ for every $x\in [0,1]$. But $f$ is clearly not constant.}
\end{ex}

A referee pointed out that our proofs in a previous version of this paper could be combined with some ideas of \cite{LiuTai} in order to yield the following generalization of Theorem \ref{main theorem}. 
For a set $E\subset\R^n$, $n\geq 2$, let us say that $f\in \widetilde{C}^{n-1,1}(E)$ if $f$ has a Taylor expansion $T_{n-1}(x; \cdot)$ of order $n-1$ at every $x\in E$ and 
\begin{equation}\label{definition of classes C tilde E}
\liminf_{y\to x}\frac{f(y)-T_{n-1}(x; y)}{|x-y|^n}>-\infty
\end{equation}
for all $x\in E$ (in particular note that $Df(x)$ exists and $f$ is continuous at $x$ for every $x\in E$). 

\begin{thm}\label{main theorem improved}
Let $E\subset\R^n$ be a set, $n\geq 2$ and $f\in \widetilde{C}^{n-1,1}(E)$. Suppose that $Df(x)=0$ for all $x\in E$.  
Then $\mathcal{L}^1 \bigl( f(E)\bigr) =0$.

\noindent The same statement holds true if $\R^n$ is replaced with an open subset of $\R^n$.
\end{thm}

Observe that if $f$ has a nonempty subdifferential of order $n$ at every $x\in E$, then $f\in \widetilde{C}^{n-1,1}(E)$. Therefore the above result clearly generalizes Theorem \ref{main theorem}.

The rest of the paper is organized as follows. In Section 2 we will recall the $C^{k-1,1}$ version of the Whitney Extension Theorem (see \cite{Whitney, Glaeser, Stein}), which will be a fundamental tool in all of our proofs, as well as a $C^{k-1,1}_{\textrm{loc}}$ version of the Kneser-Glaeser Rough Composition Theorem, and a theorem of Liu and Tai \cite{LiuTai} connecting Taylor expansions and Lusin properties of order $k$, which will be instrumental in establishing our results for the higher dimensional case ($n\geq 4$). We will use these results in Section 3, where we will provide the proofs of Theorem \ref{main theorem improved}, \ref{theorem for the proximal subdifferential in the plane}, \ref{Theorem for n=3}, and Proposition \ref{Theorem for n=1}. Finally, we include an Appendix which clarifies what is known, or at least relatively easy to get to know, about the Morse-Sard properties of vector-valued functions with Taylor expansions.

\section{Auxiliary results}

One very important tool in our proofs will be the following version of the classical Whitney Extension Theorem for functions of class $C^{m,1}$ (see \cite{Glaeser, Stein} for instance). Recall that $C^{m,1}(\R^n, \R^k)$ denotes the set of $C^{m}$ functions from $\R^n$ to $\R^k$ whose partial derivatives of order $m$ are Lipschitz.

\begin{theorem}\label{WET for functions with Lipschitz derivatives}
Let $C$ be a closed subset of $\R^n$ and $\{ f_\alpha \}_{|\alpha|\leq m}$ be a family of functions defined on $C$ and satisfying
\begin{equation} \label{derivadas de whitney}
f_\alpha(x)= \sum_{|\beta| \leq m-|\alpha|} \frac{f_{\alpha+\beta}(y)}{\beta !} (x-y)^\beta
 + R_\alpha(x,y)
\end{equation}
for all $x,y \in  C$ and all multi-indices $\alpha$ with $|\alpha| \leq m.$ Suppose that for some constant $M>0$ we have
 \begin{equation}\label{condicion whitney}
 |f_\alpha(x)| \leq M, \textrm{ and } \quad |R_\alpha(x,y)|\leq M |x-y|^{m+1-|\alpha|} \quad\ \text{for all} \quad x,y \in C
 \end{equation}
and all $|\alpha| \leq m.$ Then there exists a function $F:\R^n \longrightarrow \R$ such that:
 \begin{itemize}
  \item[(i)] $F\in C^{m,1}(\R^n, \R).$
 \item[(ii)] $D^\alpha F = f_\alpha$ on $C$ for all $|\alpha| \leq m$.
 \end{itemize}
\end{theorem}
As a matter of fact this version of  the Whitney extension theorem also holds for arbitrary sets $C$, because an obvious modification of the usual argument showing that a Lipschitz function defined on a set $D$ has a unique Lipschitz extension to the closure $\overline{D}$ of $D$, together with conditions \eqref{derivadas de whitney} and \eqref{condicion whitney}, easily imply that if $C$ is not closed then the functions $f_{\alpha}$ have unique extensions to $\overline{C}$ that also satisfy \eqref{derivadas de whitney} and \eqref{condicion whitney} on $\overline{C}$. Bearing this in mind, considering the particular case in which we have $f_{\alpha}=0$ for $|\alpha|\geq 1$, and applying the corresponding extension result to each coordinate function $f^{j}$ of a vector-valued function $f=(f^{1}, ..., f^{k})$ from a subset of $\R^n$ to $\R^k$ we immediately obtain the following.

\begin{cor}\label{consequence of Whitney ET}
Let $C$ be a (not necessarily closed) subset of $\R^n$ and $f:C\to\R^k$ be a function such that for some constant $M>0$ we have
$$
 |f(x)| \leq M, \textrm{ and } \quad |f(x)-f(y)|\leq M |x-y|^{m+1} \quad\ \text{for all} \quad x,y \in C.
$$
Then there exists a function $F\in C^{m,1}(\R^n, \R^k)$ such that $F=f$ on $C$ and $D^\alpha F = 0$ on $C$ for all $1\leq |\alpha| \leq m$.
\end{cor}

\medskip

We will also need to use a $C^{k,1}_{\textrm{loc}}$ version of the Kneser-Glaeser Theorem.
Recall that the usual Kneser-Glaeser Theorem (see \cite{AbrahamRobbin} or \cite[II.6.1]{Malgrange} for instance), whose proof relies on an application of the classical Whitney Extension Theorem, tells us that a composition of the form $f\circ g$, with $f$ of class $C^r$ and $g$ of less smoothness $C^{r-s}$, can be extended from a set $C$ to a function of class $C^r$ provided that the derivatives of $g$ up to the order $s$-th vanish on $C$. This kind of result also holds true for the classes $C^{k,1}_{\textrm{loc}}$. Recall that a function $f$ belongs to $C^{k, 1}_{\textrm{loc}}$ provided $f$ is $k$ times continuously differentiable and the partial derivatives $D^{\alpha}f$ are locally Lipschitz for all multi-indices $\alpha$ of order $k$ (or equivalently for all multi-indices $\alpha$ with $|\alpha|\leq k$).

\begin{theorem}[Kneser-Glaeser]\label{Kneser Glaeser}
Let $W\subset\R^m$ and $V\subset\R^n$ be open sets; $A^{*}\subset W$ and
$A\subset V$, with $A$ closed relative to $V$, $f:V\to\R^p$ of class $C^{r, 1}_{\textrm{loc}}$ on $V$ and $s$-flat on $A$, $g:W\to V$ of class $C^{r-s, 1}_{\textrm{loc}}$ with $g(A^{*})\subset A$. Then there is a map $H:W\to\R^{p}$ of class $C^{r, 1}_{\textrm{loc}}$ satisfying:
\begin{enumerate}
\item $H(x)=f(g(x))$ for $x\in A^{*}$;
\item $H$ is $s$-flat on $A^{*}$.
\end{enumerate}
\end{theorem}

\noindent {\em Sketch of proof}. We will follow the proof of the classical version of the Kneser-Glaeser theorem that appears in 
\cite[pages 35-37]{AbrahamRobbin}, explaining what small additions we need to make in order to obtain the present version. We may assume that $A^{*}$ and $A$ are compact (the general case follows from this particular situation via standard arguments with partitions of unity). Our starting point is the Composite Mapping Formula: if
suppose for the moment that $g$ is $C^r$, then we would have
$$
D^j(f\circ g)(x)=\sum_{q=1}^j
\sum_{i_1,\dots ,i_q}
\sigma_j(i_1,\dots ,i_q)D^qf(g(x))\circ
\bigl( D^{i_1}g(x),\dots , D^{i_q}g(x) \bigr)
$$
for every $j\leq r$, where $i_1,\dots ,i_q$ are positive integers satisfying $i_1+\dots +i_q=j$. 
If $x\in A^*$, then $g(x)\in A$ and we have that $D^q f(g(x))=0$ provided that $q\leq s$, since $f$ is $s$-flat.
Hence the sum runs from $q=s+1$ to $j$, and $i_1,\dots ,i_q\leq r-s$ necessarily. This implies that 
$ D^{i_1}g(x),\dots , D^{i_q}g(x)$, and consequently
$$
D^qf(g(x))\circ \bigl( D^{i_1}g(x),\dots , D^{i_q}g(x) \bigr)
$$
are well defined. This allows us to define, if $x\in A^*$, $h_0(x)=(f\circ g)(x)$, $h_k(x)=0$ if $k\leq s$ and 
$$
h_k(x)=\sum_{q=s+1}^k
\sum_{i_1,\dots ,i_q}
\sigma_j(i_1,\dots ,i_q)D^qf(g(x))\circ
\bigl( D^{i_1}g(x),\dots , D^{i_q}g(x) \bigr)
$$
if $s+1\leq k\leq r$.

Now, for $k<r$, the Taylor Formula
with integral remainder
$$
F(x)=\sum_{i=0}^m\frac{1}{i!}D^{i}F(y)(x-y)^{i}+
$$
$$
+\Bigl( \int_0^1\frac{(1-t)^{m-1}}{(m-1)!}[D^{m}F(y+t(x-y))-D^{m}F(y)]dt\Bigr) (x-y)^{m}=
$$
$$
:=\sum_{i=0}^m\frac{1}{i!}D^{i}F(y)(x-y)^{i}+\mathcal{R}_F^m(x,y)(x-y)^m,
$$
applied to $f$, $g$ and their derivatives,  
allows one to deduce the following Taylor-like formula for the functions $h_k$:
$$
h_k(x)=\sum_{j=0}^{r-k}\frac{1}{j!}h_{k+j}(y)(x-y)^j+R_k(x,y)
$$
(see \cite[pages 35-37]{AbrahamRobbin} for details),
where $R_k(x,y)$ is a sum of terms of the 
forms: 
\begin{enumerate}
  \item $A(x-y)^j$, with  $A$ $j$-linear, $j>r-k$.
  \item $\mathcal{R}_{D^qf}^m(g(x),g(y))(g(x)-g(y))^m\bigl( \dots \bigr)$, $m=r-q$.
  \item $D^qf(g(x))\bigl( \dots ,\mathcal{R}_{D^{i_l}g}^{r-s-i_l}(x,y)(x-y)^{r-s-i_l},\dots \bigr)$, $i_l\leq k-q+1$
\end{enumerate}
The terms of type $(1)$ are obviously $O(|x-y|^{r-k+1})$:
$$
|A(x-y)^j|\leq ||A|||y-x|^j\leq ||A|||x-y|^{r-k+1}.
$$
The terms of type $(2)$ satisfy
$$
\Bigl| \mathcal{R}_{D^qf}^m(g(x),g(y))(g(x)-g(y))^m\bigl( \dots \bigr) \Bigr| \leq 
$$
$$
\leq K_0\Bigl| \Bigl| \mathcal{R}_{D^qf}^m(g(x),g(y))(g(x)-g(y))^m \Bigr| \Bigr|
$$
(because the hidden arguments are uniformly bounded, recall that $g$ is $C^{r-s}$ and $A^{*}$ is compact), which we can join with the inequalities  
$$
K_0\Bigl| \Bigl| \mathcal{R}_{D^qf}^m(g(x),g(y))(g(x)-g(y))^m \Bigr| \Bigr|  \leq 
$$
$$
\leq K_0\int_0^1\frac{(1-t)^{m-1}}{(m-1)!} \bigl| \bigl|
D^r f(g(y)+t(g(x)-g(y)))-D^r fg((y))\bigr| \bigr| dt |x-y|^{r-q}\leq
$$
$$
\leq K_0\int_0^1\frac{(1-t)^{m-1}}{(m-1)!} tK_1|g(x)-g(y)| dt |x-y|^{r-q}\leq
$$
$$
\leq K_0\int_0^1\frac{(1-t)^{m-1}}{(m-1)!} tK_1K_2|x-y| dt |x-y|^{r-q}\leq 
$$
$$
\leq K|x-y|^{r-q+1}\leq K|x-y|^{r-k+1}
$$
since $D^rf$ is $K_1$-Lipschitz for some constant $K_1$, $g$ s $K_2$-Lipschitz for some constant $K_2$, and $q\leq k$.

Finally, for the terms of type $(3)$, we have
$$
\Bigl| D^qf(g(x))\bigl( \dots ,\mathcal{R}_{D^{i_l}g}^{r-s-i_l}(x,y)(x-y)^{r-s-i_l},\dots \bigr) \Bigr| \leq
$$
$$
\tilde{K}_0 \Bigl| \Bigl| \mathcal{R}_{D^{i_l}g}^{r-s-i_l}(x,y)(x-y)^{r-s-i_l} \Bigr| \Bigr| \leq
$$
$$
\leq \tilde{K}_0\Bigl( \int_0^1\frac{(1-t)^{r-s-i_l-1}}{(r-s-i_l-1)!}t\tilde{K}|x-y|dt\Bigr) |x-y|^{r-s-i_l}=
$$
$$
=K |x-y|^{r-s-i_l+1}\leq K|x-y|^{r-k+1}
$$
since $D^{r-s}g$ is $\tilde{K}$-Lipschitz for some constant $\tilde{K}$, and $s+i_l\leq k$. Here $\tilde{K}_0$ is an uniform bound
for $D^qf(g(x))$, and the derivatives and integral remainders of $g$, near $x$).

Summing up, we have 
$$
R_k(x,y)=O\left(|x-y|^{r-k+1}\right) \textrm{ for all } x, y\in A^{*},
$$
and consequently by 
Theorem \ref{WET for functions with Lipschitz derivatives} there exists 
a $C^{r,1}$ function $H:\mathbb{R}^m\to \mathbb{R}^p$, such that $D^kH(x)=h_k(x)$ for every $x\in A^*$. In
particular $H(x)=f(g(x))$ for every $x\in A^*$, and $H$ is $s$-flat on $A^*$.
\qed

\medskip

In the proof of Theorem \ref{main theorem improved}, in order to deal with a case which will only be present in dimensions $n\geq 4$ (see Lemma \ref{Derivadas no nulas} below), we will need to combine the preceding version of the Kneser-Glaeser theorem with the following important result of Liu and Tai about Taylor polynomials and Lusin properties of order $k$.

\begin{theorem}[Liu-Tai, see \cite{LiuTai}]\label{LiuTai}
For a measurable function $u$ defined on a measurable set $D$ of $\R^n$, the following statements are equivalent:
\begin{enumerate}
\item $u$ has the Lusin property of order $k$ on $D$.
\item $u$ has an approximate $(k-1)$-Taylor polynomial at almost every point of $D$.
\item $u$ is approximately differentiable of order $k$ at almost every point of $D$.
\end{enumerate}
\end{theorem}
Recall that $u$ is said to have the Lusin property of order $k$ provided that for every $\varepsilon>0$ there exists a function $g\in C^{k}(\R^n)$ such that $$\mathcal{L}^{n}\left(\{x\in D \, : \, u(x)\neq g(x)\}\right)<\varepsilon.$$ 
Also recall that $\textrm{aplim}_{y\to x} v(y)=\alpha$ means that the set $\{y\in D : |v(y)-\alpha|\leq\varepsilon\}$ has density one at $x$ for every $\varepsilon>0$, and that $\textrm{aplimsup}_{y\to x}v(y)$ is the infimum of all $\beta\in\R$ such that the set $\{y\in D: u(y)>\beta\}$ has density zero at $x$.
Then one says  that $u$ has an approximate Taylor polynomial of order $k-1$ at $x$ if there exists a polynomial $p(x; y)$ of order at most $k-1$ such that 
$$
\textrm{aplimsup}_{y\to x}\frac{|u(y)-p(x ; y)|}{|y-x|^{k}}<\infty.
$$
Similarly, $u$ is said to be approximately differentiable of order $k$ at $x$ provided there exists a polynomial $p(x; y)$ of order at most $k$ such that
$$
\textrm{aplim}_{y\to x}\frac{|u(y)-p(x ; y)|}{|y-x|^{k}}=0.
$$
Observe that, according to the definitions above, if $\partial^{n}f(x)\neq\emptyset$ then $f$ is approximately differentiable of order $n-1$ at $x$, and in particular $f$ has an approximate $(n-2)$-Taylor polynomial (but not necessarily an approximate $(n-1)$-Taylor polynomial) at $x$.

\section{Proofs of the main results}

Let us start by giving the easy {\bf Proof of Proposition \ref{Theorem for n=1}.}
We may assume that $C_f\subset I$, where $I$ is an interval of length $1$.
Given $\varepsilon >0$, we define for every $j$ the closed set
$$
D_j=\{ x\in \R: f(x+t)\geq f(x)-\varepsilon |t|, \ \hbox{for every} \ |t|<\frac{1}{j}\}.
$$
The sequence $\{ D_j\}$ is increasing and satisfies $C_f\subset \cup_jD_j$ since
$$
\liminf_{t\to 0}\frac{f(x+t)-f(x)}{|t|}\geq 0
$$
for every $x\in C_f$.

We split the interval $I$ into $j$ intervals $I_k$ of length $\frac{1}{j}$. For every $x,y\in D_j\cap I_k$
we have
$$
|f(x)-f(y)|\leq \varepsilon |x-y|\leq \frac{\varepsilon}{j},
$$
hence
$$
\mathcal{L}^{1} \bigl( f(D_j\cap I_k)\bigr) \leq \frac{\varepsilon}{j},
$$
and consequently $\mathcal{L}^{1} \bigl( f(D_j)\bigr) \leq \varepsilon$.
We deduce that
$$
\mathcal{L}^{1}\bigl( f(C_f)\bigr) \leq \mathcal{L}^{1} \bigl( \cup_j f(D_j)\bigr) =\lim_j\mathcal{L}^{1} \bigl( f(D_j)\bigr) \leq \varepsilon ,
$$
and therefore $\mathcal{L}^{1} \bigl( f(C_f)\bigr) =0$. \,\,\, $\Box$

\medskip

Let us now proceed with the {\bf Proof of Theorem \ref{main theorem improved}.} For every $x\in E$  we will also denote $T_{n-1}(x;\cdot)$, the Taylor expansion of order $n-1$ of the function $f$ at $x$,
by $f(x)+P_x$, with $P_x=P_x^1+\dots P_x^{n-1}$, where $P_x^k$ is the $k$-homogeneous part of $P_x$. Recall that, by assumption, $Df(x)=0$ for every $x\in E$. We consider the following decomposition
$$
E=A\cup B\cup \bigl( E\setminus (B\cup A)\bigr),
$$
where
$$
A=\{ x\in E: P_x\equiv 0\}, \ B=\{ x\in E: P_x=P_x^{n-1}\not\equiv 0\}.
$$

Our first goal is proving the following.

\begin{lemma}\label{polinomio de taylor 0}
We have $\mathcal{L}^{1} \bigl( f(A)\bigr) =0$.
\end{lemma}
\begin{proof}
It is clear that
$$
A=\bigcup_{m=0}^{\infty}A_m,
$$
where
$$
A_m=\{ x\in A: \liminf_{y\to x}\frac{f(y)-T_{n-1}(x;y)}{|x-y|^{n}}\geq -m\}.
$$
Therefore it is enough to prove that $\mathcal{L}^{1} (f(A_m))=0$.
We denote
$$
A_k^m=\{ x\in A: f(x+h)\geq f(x)-(m+1)|h|^n \ \hbox{if} \ |h|<\frac{1}{k} \} \cap \bar{B}(0,k).
$$
It is clear that
$$
A_m\subset \bigcup_{k=1}^{\infty}A_k^m.
$$
Hence it is enough to prove $\mathcal{L}^{1}\bigl( f(A^m_k)\bigr) =0$. We cover $A^m_k$ by a countable collection of closed cubes $Q$ of diameter less or equal
than $\frac{1}{k}$. Let us denote $D=Q\cap A^m_k$ in the rest of the argument. Our aim is to prove that $\mathcal{L}^{1} (f(D))=0$. We have
\begin{equation}\label{control oerden n}
|f(x)-f(y)|\leq (m+1)|x-y|^n
\end{equation}
for every $x,y\in D$.  By using this inequality and regarding $f$ for a moment as a function defined just on the set $D$,
we may extend it to a $C^{n-1,1}$ function $\tilde{f}=\R^n\to \R$, with $D^l\tilde{f}(x)=0$ for
$x\in D$, $l=0,1,\dots ,n-1$, by means of Corollary \ref{consequence of Whitney ET}. By Bates's version \cite{Bates} of the Morse-Sard theorem for functions in the class $C^{n-1,1}(\R^n)$ (or by any of its generalizations \cite{DePascale}, \cite{BourKoKris2}, \cite{KorobkovKristensen} and \cite[Theorem 1.2]{AFG}) we then have $\mathcal{L}^{1} (\tilde{f}(D))=0$, hence $\mathcal{L}^{1} (f(D))=0$ too.
\end{proof}

Let us observe that Lemma \ref{polinomio de taylor 0} has the following consequence.

\begin{cor}\label{teorema para n=2}
Theorem \ref{main theorem} is true for $n=2$.
\end{cor}
\begin{proof}
Note that $C_f=A$ for $n=2$.
\end{proof}

We also note that exactly the same argument as in the proof of {\bf Lemma \ref{polinomio de taylor 0}  provides a proof of Theorem \ref{theorem for the proximal subdifferential in the plane}.}
Since the case $n=2$ is already dealt with, from now we will assume that $n\geq 3$. 

In the following two lemmas we will show that $\mathcal{L}^1 \bigl( f(B)\bigr) =0$. Recall that if $x\in B$ then $P_x=P_x^{n-1}\not\equiv 0$.
\begin{lemma}\label{n impar}
For $n\geq 3$ odd, we have
$\mathcal{L}^1 \bigl( f(B)\bigr) =0$.
\end{lemma}
\begin{proof}
We may assume that for every $x\in B$ there exists $e_x\in \mathbb{S}^{n-1}$ such that $P_x^{n-1}(e_x)<0$.
Let $\{ e_i\}_{i\in \mathbb{N}}$ be dense in $\mathbb{S}^{n-1}$. If we define
$$
B_{i,m}=\{ x\in B: P_x^{n-1}(e_i)\leq -\frac{1}{m}, ||P_x^{n-1}||\leq m \}
$$
we have that
$$
B=\bigcup_{i,m=1}^{\infty}B_{i,m}.
$$
Thus it is enough to prove $\mathcal{L}^1 \bigl( f(B_{i,m})\bigr) =0$.
Let $0<\varepsilon <\frac{1}{m}$. We define
$$
D_j=\{ x\in E: |f(x+h)-f(x)-P_x^{n-1}(h)|\leq \varepsilon |h|^{n-1}\ \hbox{if} \ \ |h|<\frac{1}{j}\} .
$$
The sequence $\{D_j\}$ is increasing. For every $x\in B_{i,m}$, we have that
$$
\lim_{h\to 0}\frac{f(x+h)-f(x)-P_x^{n-1}(h)}{|h|^{n-1}}=0,
$$
and consequently there exists $j$ such that $x\in D_j$. Hence $B_{i,m}\subset  \cup_jD_j$.

If $x\in D_j$ and $t\in (-\frac{1}{j},\frac{1}{j})$, $t\neq 0$, then
\begin{equation}\label{maxima in cubes}
f(x+te_i)\leq f(x)+t^{n-1}(P_x(e_i)+\varepsilon )\leq f(x)+t^{n-1}(-\frac{1}{m}+\varepsilon )<f(x).
\end{equation}
(Note that if $P_x^{n-1}(e_x)>0$ instead of $P_x^{n-1}(e_x)<0$, we will have local minima instead of local maxima, and the subsequent
arguments work as well.)

Let $\{C_r^j\}_r$ be a covering of $D_j$ by closed cubes with one edge parallel to $e_i$ and length equal to  $\frac{1}{\sqrt{n}j}$.
For every line $L$ parallel to $e_i$, equation \eqref{maxima in cubes} implies that
$$
\sharp \bigl( D_j\cap C_r^j\cap L\bigr) \leq 1
$$
and if $x\in  D_j\cap C_r^j\cap L$ then $f_{|_{L\cap C^{j}_{r}}}$ attains a strict maximum at $x$.
Let $F_r^j=\pi (C_r^j\cap D_j)$ the projection of $C_r^j\cap D_j$ on $[e_i]^{\perp}$ (the orthogonal complement of the line spanned by $e_i$). Define $g:F_r^j\to \R$ by
$g(\bar{x})=\max \{ f(\bar{x}+te_i): \bar{x}+te_i\in C_r^j\}$. 
Now, by means of Corollary \ref{consequence of Whitney ET}, we may extend $g$ as a $C^{n-2,1}$ function to the whole
$\R^{n-1}$ (we are identifying $[e_i]^{\perp}$ and $\R^{n-1}$),
with derivatives $D^lg(\bar{x})=0$, $l=1,\dots n-2$, for every $\bar{x}\in F_r^j$. Indeed, we have, for every $\bar{x},\bar{y}\in F_r^j$, that
$$
g(\bar{x})-g(\bar{y})=f(x)-g(\bar{y})\leq f(x)-f(z) \leq
$$
$$
 \leq |f(z)-f(x)|\leq (m+1)|x-z|^{n-1}=(m+1)|\bar{x}-\bar{y}|^{n-1}
$$
where $z$ satisfy $\pi (z)=\bar{y}$ and $\bar{x}-\bar{y}=x-z$,
because
$$
|f(x+h)-f(x)|\leq (m+1)|h|^{n-1}
$$
for every $x\in D_j$ provided that $|h|<\frac{1}{j}$. Hence it is clear that the conditions of Corollary \ref{consequence of Whitney ET} are met. Therefore, by Bates's version of the Morse-Sard theorem we have $\mathcal{L}^{1} (\bigl( g(F_k^j)\bigr)=0$, and since
$f(D_j\cap C_k^j) = g(F_k^j)$ we conclude that $\mathcal{L}^{1} \left(\bigl( f(D_j\cap C_k^j) \bigr)\right)=0$ too.
\end{proof}

On the other hand, we have

\begin{lemma}\label{n par}
If $n\geq 4$ is even, then
$\mathcal{L}^1 \bigl( f(B)\bigr) =0$.
\end{lemma}
\begin{proof}
Recall that every $x\in B$ satisfies
$$
\lim_{h\to 0}\frac{f(x+h)-f(x)-P_x^{n-1}(h)}{|h|^{n-1}}= 0
$$
with $P_x^{n-1}\not\equiv 0$. By considering
$B=\bigcup_{j=1}^{\infty}\{x\in B : \|P_{x}^{n-1}\|\geq 1/j\}$, we may assume without loss of generality that for every
$x\in B$, the corresponding $P_x^{n-1}$ satisfy $\|P_x^{n-1}\|\geq 2c$ for a fixed positive constant $c$.

Let $\{T_i\}$ be a dense sequence in the space of $(n-1)$-homogeneous polynomials on $\R^n$. Let $0<\varepsilon <\frac{c}{2}$.
We define
$$
B_{j,i}=\{ x\in \R^n: |f(x+h)-f(x)-T_i(h)|\leq \varepsilon |h|^{n-1}, |h|\leq \frac{1}{j}\}.
$$
For every $x\in B$, there exist $j$ and $T\in \mathcal{P} (^{n-1}\R^n)$ such that
$$
|f(x+h)-f(x)-T(h)|\leq \frac{\varepsilon}{2} |h|^{n-1}
$$
if $|h|\leq \frac{1}{j}$. Let $T_i$ be such that $||T_i-T||<\frac{\varepsilon}{2}$; note that
$||T_i||>2c-\frac{\varepsilon}{2}$ necessarily. We have
$$
|f(x+h)-f(x)-T_i(h)|\leq \frac{\varepsilon}{2} |h|^{n-1}+|T(h)-T_i(h)|\leq \varepsilon |h|^{n-1}
$$
if  $|h|\leq \frac{1}{j}$, hence $x\in B_{j,i}$. That is
$$
B\subset \bigcup_{j,i=1}^{\infty}B_{j,i}.
$$
Consequently, it is enough to prove that
$$
\mathcal{L}^{1} \bigl( f(B_{j,i})\bigr) =0 \ \ \hbox{for every} \ i,j .
$$
In the rest of the argument we denote by $D$ one of such $B_{j,i}$, that is
$$
D=\{ x\in \R^n:|f(x+h)-f(x)-T_i(h)| \leq \varepsilon |h|^k \textrm{ if } |h|\leq \frac{1}{j}\}.
$$
It is clear that $D$ is closed. Find $e_i\in \mathbb{S}^{n-1}$ such that $T_i(e_i)=\frac{3}{2}c$, and let $L$ be a line parallel to $e_i$.
We may split $D$ in a countable collection of sets with diameter less than or equal
to $\frac{1}{j}$, and thus we may assume that $D$ itself has diameter less or equal
than $\frac{1}{j}$. For every $e\in \mathbb{S}^{n-1}$ satisfying $T_i(e)\geq c$, we consider lines $L_e$ parallel to $e$.
\begin{claim}
We have $ \sharp (D\cap L_e)\leq 1$, and in particular $ \sharp (D\cap L)\leq 1$.
\end{claim}
Assume for a moment that the Claim is true.
Let $F=\pi (D)$ the projection of $D$ on $[e_i]^{\perp}$. Note that $\pi :D \to F$ is a homeomorphism since $D$ is compact and we are assuming the Claim.
We define $g:F\to \R$ by
$g(\bar{x})=f(\pi ^{-1}(\bar{x}))$. 

Again, we are going to extend $g$ to a $C^{n-2,1}$ function on
$\R^{n-1}$  with derivatives $D^lg(\bar{x})=0$, $l=1,\dots ,n-2$.
Let $\bar{x},\bar{y}\in F$, $x=\bar{x}+t_xe_i\in D$ and $y=\bar{y}+t_ye_i\in D$, and let
$\alpha >0$ be such that $T_i(e)\geq c$ provided that the angle $\widehat{e,e_i}$ between $e$ and $e_i$ satisfies $\widehat{e,e_i}<\alpha$. We have that
$$
|\bar{x}-\bar{y}|\geq \sin \alpha |x-y|.
$$
Indeed, otherwise $\widehat{e_i,e_{xy}}<\alpha$, or $\widehat{e_i,-e_{xy}}<\alpha$, for $e_{xy}=\frac{y-x}{|y-x|}$,
which implies $x=y$ (by the Claim and the fact that $x,y$ belong to the line $L_{e_{xy}}$).
Consequently, we have
$$
|g(\bar{x})-g(\bar{y})|=|f(x)-f(y)|\leq (||T_i||+\varepsilon )|x-y|^{n-1}\leq
\frac{(||T_i||+\varepsilon )}{(\sin \alpha )^{n-1}} |\bar{x}-\bar{y}|^{n-1}
$$
for every $\bar{x},\bar{y}\in F$. This inequality and Corollary \ref{consequence of Whitney ET} allow us to extend $g$ to $\R^{n-1}$ as a $C^{n-2,1}$
function, with null derivatives up to the order $n-2$ at the points of $F$, and we may then conclude the proof as in Lemma \ref{n impar}.

\medskip

It only remains to prove the Claim. First, we observe that for every $x\in D$ and $|t|\leq \frac{1}{j}$, we have
$$
|f(x+te)-f(x)-T_i(e)t^{n-1}| \leq \varepsilon |t|^{n-1},
$$
or equivalently
$$
-\varepsilon |t|^{n-1}\leq f(x+te)-f(x)-T_i(e)t^{n-1} \leq \varepsilon |t|^{n-1}.
$$
Hence
$$
f(x+te)\in \bigl( f(x)+(T_i(e)-\varepsilon )t^{n-1},f(x)+(T_i(e)+\varepsilon )t^{n-1} \bigr)
$$
for $t\in (0, 1/j)$, while
$$
f(x+te)\in \bigl( f(x)+(T_i(e)+\varepsilon )t^{n-1},f(x)+(T_i(e)-\varepsilon )t^{n-1} \bigr)
$$
for $t\in (-1/j,0)$. Now suppose, seeking a contradiction, that we had $x, y\in D\cap L$, $x\neq y$ and $r=|y-x|>0$. Then we may assume that $y=x+|y-x|e$ (and $x=-|y-x|e$).
Set $z=x+\frac{r}{2}e=y-\frac{r}{2}e$. We have
\begin{eqnarray*}
& & f(z)<f(x)+(T_i(e)+\varepsilon )(\frac{r}{2})^{n-1}<f(x)+\frac{3}{2}\frac{r^{n-1}}{2^{n-1}}T_i(e),\\
& &
f(z)>f(y)-(T_i(e)+\varepsilon )(\frac{r}{2})^{n-1}, \textrm{ and } \\
& &
f(y)>f(x)+(T_i(e)-\varepsilon )r^{n-1}>f(x)+\frac{T_i(e)}{2}r^{n-1}.
\end{eqnarray*}
From these inequalities we obtain
$$
f(x)+\frac{T_i(e)}{2}r^{n-1}-(T_i(e)+\varepsilon )(\frac{r}{2})^{n-1}<f(z)<f(x)+\frac{3}{2}\frac{r^{n-1}}{2^{n-1}}T_i(e),
$$
which implies
$$
\frac{T_i(e)}{2}r^{n-1}<3\frac{r^{n-1}}{2^{n-1}}T_i(e).
$$
Hence $1<\frac{6}{2^{n-1}}$, which is impossible since $n\geq 4$.
\end{proof}

Observe that for $n=3$ we have $C_f=B\cup A$, and consequently the same proof as above establishes Theorem \ref{Theorem for n=3}.

It only remains to show that $\mathcal{L}^1 \Bigl( f\bigl( E\setminus (B\cup A)\bigr) \Bigr) =0$. Note that if $n\leq 3$ then $E=A\cup B$ and this is trivially true. Therefore in the sequel we will assume that $n\geq 4$. 

\begin{lemma}\label{Derivadas no nulas} We have that
$
\mathcal{L}^1 \Bigl( f\bigl( E\setminus (B\cup A)\bigr) \Bigr) =0.
$
\end{lemma}
\begin{proof}
We will make use of the arguments in the proof of Liu-Tai's Theorem \ref{LiuTai} given in \cite[pages 193-194]{LiuTai}, which in turn employ the following lemma (see \cite[Lemma 2.1]{Campanato} for a proof).

\begin{lemma}[De Giorgi]\label{De Giorgi}
Let $V$ be a measurable subset of the ball $B(x,r)$ of $\R^n$ such that $\mathcal{L}^{n}(V)\geq A r^n$ for some constant $A>0$. Then for each $k\in\N$ there exists a constant $C=C(n,k, A)>0$ (depending only on $n$, $k$ and $A$) such that
$$
|D^{\alpha}p(x)|\leq \frac{C}{r^{n+|\alpha|}}\int_{V}|p(y)|dy
$$
for all polynomials $p$ of degree at most $k$.
\end{lemma}
Let $T_{n-2}(x; \cdot)$ be the Taylor polynomial of order $n-2$ of $f$, centered at $x$; of course, $T_{n-2}(x; \cdot)$ is obtained from $T_{n-1}(x; \cdot)$ by discarding $P_{x}^{n-1}$, the $(n-1)$-homogeneous term of $T_{n-1}(x; \cdot)$. Since every ordinary limit is also an approximate limit, the proof of \cite[Theorem 1]{LiuTai} shows that the coefficients of the polynomials $T_{n-1}(x; y)$, which keeping Liu-Tai's notation we will denote by $f_{\alpha}(x)$, are measurable functions of $x$ whenever $E$ is measurable (this can also be deduced from the statement of Liu-Tai's theorem). As a matter of fact, we will not need to use measurability of the coefficients $f_{\alpha}(x)$; we just mention that this is so because it may be an interesting and useful fact to know for those readers who are not already acquainted with the arguments of \cite{LiuTai}.

Because $f$ has a Taylor expansion of order $n-1$ at each point $x$ of $E$, we can write $$E=\bigcup_{j=1}^{\infty}E_{j},$$
where
$$
E_{j}:=\left\{ x\in E \, : \, \frac{|f(y)-T_{n-2}(x;y)|}{|y-x|^{n-1}}\leq j \textrm{ for all } y \textrm{ with } 0<|y-x|\leq\frac{1}{j}\right\}\cap D_j,
$$
with
$$
D_j=\left\{ x\in E \, : \, |f_{\alpha}(x)|\leq j, |\alpha|\leq k-2\right\}.
$$
We claim that the arguments of Liu and Tai's in \cite[page 193]{LiuTai} imply that
$$
|D^{\alpha}f(y)-D^{\alpha}T_{n-2}(x; y)|\leq M j|y-x|^{n-1-|\alpha|} \eqno(*)
$$
for all $x,y\in E_j$ with $|x-y|\leq 1/j$ and all multi-indices $\alpha$ of order $|\alpha|\leq n-2$, where $M$ is a constant depending only on $n$. 

Assume for the moment that this claim is true, and let us see how the proof of Lemma \ref{Derivadas no nulas} can be completed. Then can apply Theorem \ref{WET for functions with Lipschitz derivatives} (for not necessarily closed sets, see the remark after its statement) in order to find a function $g_{j}\in C^{n-2,1}(\R^n)$ such that the restriction of $g_j$ to $E_j$ coincides with $f$, and the restriction of each partial derivative $D^{\alpha}g_{j}$ to $E_j$ coincides with $\partial^{\alpha}f$, for all multi-indices $\alpha$ of the order $|\alpha|\leq n-2$. This obviously implies (denoting the set of critical points of a function $\varphi$ by $C_{\varphi}$) that
$$
f(E)\subseteq\bigcup_{j=1}^{\infty}g_{j}(C_{g_{j}}).
$$
Therefore we may and do assume in the sequel that $f$ is of class $C^{n-2,1}$.
Now, for each $x\in C:=E\setminus (B\cup A)$,
let $k_x$ be the smallest index such that $P^{k_x+1}_x\not\equiv 0$ but $P^l_x\equiv 0$ for $l=1,\dots ,k_x$. Note that
$1\leq k_x\leq n-3$ necessarily. We may thus split $C$ into $n-3$ subsets on each of which $k_x$ is constant, and then assume without loss of generality
that $k_x=k$, a constant, for every $x\in C$. In particular, $f$ is $k$-flat on $C$. By the implicit function theorem and local compactness, we may write
$$
C\subseteq \bigcup_{j=1}^{\infty}M_j,
$$
where the $M_j$ are manifolds of dimension $n-1$ parametrized by functions $h_j:W_j\subset\R^{n-1}$ of class $C^{n-2-k,1}_{\textrm{loc}}$. Hence we may further assume that $C$ is just one of these manifolds, say $C=h(W)$, $W\subseteq\R^{n-1}$, with $h\in C^{n-2-k,1}_{\textrm{loc}}(\R^{n-1})$. Let us denote $C^{*}=g^{-1}(C)$. Then, by Theorem \ref{Kneser Glaeser}, there exists $H\in C^{n-2, 1}_{\textrm{loc}}(\R^{n-1})$ such that $H(x)=f(h(x))$ for every $x\in C^{*}$ and $H$ is $k$-flat on $C^{*}$, and in particular $DH=0$ on $C^{*}$. This means that 
$$
f(C)\subset H(\{x: DH(x)=0\}).
$$
But, according to Bates's theorem, $$\mathcal{L}^{1}\left(H(\{x: DH(x)=0\})\right)=0.$$ Therefore
$\mathcal{L}^{1}(f(C))=0$ as well, and we are done.

Now let us see how Liu and Tai's arguments in \cite[page 193]{LiuTai} allow us to establish $(*)$. Denote
$$
\rho=\frac{\mathcal{L}^{n}(B(x, |y-x|)\cap B(y, |y-x|))}{|y-x|^n}, \,\,\, x, y\in\R^n, x\neq y,
$$
and observe that $\rho$ is independent of $x, y$. Now fix $j\in\N$ and consider two different points $x, y\in E_j$ with $|x-y|\leq 1/j$, and define
$$
V(x,y,j)=B(x, |x-y|)\cap B(y, |x-y|).
$$ 
If $z\in V(x,y,j)$ we have, for the polynomial $q(z)=T_{n-2}(y;z)-T_{n-2}(x;z)$, that
\begin{eqnarray*}
& & |q(z)|\leq |T_{n-2}(x;z)-f(z)|+|f(z)-T_{n-2}(y;z)|\\
& & \leq j \left(|z-x|^{n-1}+|y-z|^{n-1}\right)\leq 2j|x-y|^{n-1}.
\end{eqnarray*}
Then we can apply De Giorgi's Lemma, with $V=V(x,y,j)$ and $r=|x-y|$, to obtain
$$
|D^{\alpha}q(y)|=|f_{\alpha}(y)-D^{\alpha}T_{n-2}(x;y)|\leq\frac{C}{r^{n+|\alpha|}}\int_{V(x,y,j)}|q(z)|dz\leq 2j\rho C r^{n-1-|\alpha|}.
$$
This shows $(*)$.
\end{proof}

\medskip

\section*{Appendix}

In a preliminary version of this paper we included a short proof of the following result.
\begin{theorem}\label{improvement of a theorem of AFG1}
Let $n\geq m$ be positive integers, $k:=n-m+1$ and let $f:\R^n\to\R^m$ be such that
\begin{enumerate}
\item $f\in C^{k-1}(\R^n, \R^m)$;
\item $\limsup_{h\to 0}\frac{|f(x+h)-f(x)-Df(x)(h) - ... - \frac{1}{(k-1)!} D^{k-1}f(x)(h^{k-1})|}{|h|^k}<\infty$ for every $x\in\R^n$.
\end{enumerate}
Then $\mathcal{L}^{m}\left(f(C_f) \right)=0$, where
$C_f:=\{ x\in\R^n : \textrm{rank}\left(Df(x)\right) < m\}$.

\noindent The same statement holds true if $\R^n$ is replaced with an open subset of $\R^n$. 
\end{theorem}
(In the special case $n=m$ and $k=1$, the above statement simply says that a Stepanoff function has the Morse-Sard property, a fact which is well known.)

An equivalent version of this result  was established in \cite{AFG} as a corollary to an abstract version of the Morse-Sard Theorem; note that the assumptions of the above result are equivalent to those of \cite[Theorem 1.1]{AFG} thanks to a result of Liu and Tai \cite[Theorem 2]{LiuTai}.

A referee of this paper pointed out that a very short proof of Theorem \ref{improvement of a theorem of AFG1} can be obtained by using some ideas of the proof of \cite[Theorem 1]{LiuTai}. We next offer the gist of his argument.

Under the standing hypothesis, one can write $$\R^n=\bigcup_{j=1}^{\infty}E_{j},$$
where
$$
E_{j}:=\left\{ x\in\R^n \, : \, \frac{|f(y)-T_{k-1}(x;y)|}{|y-x|^k}\leq j \textrm{ for all } y \textrm{ with } 0<|y-x|\leq\frac{1}{j}\right\}
$$
and $T_{k-1}(x; \cdot)$ is the Taylor polynomial of order $k-1$ of $f$, centered at $x$.
Each set $E_j$ is closed, and by following the lines of the proof of \cite[Theorem 1]{LiuTai} it can be shown that
$$
|\partial^{\alpha}f(y)-\partial^{\alpha}T_{k-1}(x; y)|\leq C j|y-x|^{k-|\alpha|}
$$
for all $x,y\in E_j$ with $|x-y|\leq 1/j$ and all multi-indices $\alpha$ of order $|\alpha|\leq k-1$. Then one can apply Theorem \ref{WET for functions with Lipschitz derivatives} (the $C^{k-1,1}$ version of the Whitney Extension Theorem) to find a function $g_{j}\in C^{k-1,1}(\R^n, \R^m)$ such that the restriction of $g_j$ to $E_j$ coincides with $f$, and the restriction of each partial derivative $\partial^{\alpha}g_{j}$ to $E_j$ coincides with $\partial^{\alpha}f$, for all multi-indices $\alpha$ of the order $|\alpha|\leq k-1$. This obviously implies (denoting the set of critical points of a function $\varphi$ by $C_{\varphi}$) that
$$
f(C_f)\subseteq\bigcup_{j=1}^{\infty}g_{j}(C_{g_{j}}).
$$
However, according to Bates's theorem \cite{Bates} the sets $g_{j}(C_{g_{j}})$ are of measure $0$ in $\R^m$. Therefore so is $f(C_f)$. \qed

\medskip

In fact, because the $C^{k-1,k}$ version of Whitney's extension theorem holds for not necessarily closed sets as well (see the remark after Theorem \ref{WET for functions with Lipschitz derivatives} above), the preceding argument can also be arranged to show that condition $(1)$ of Theorem \ref{improvement of a theorem of AFG1} can be dispensed with: namely, one also has the following.

\begin{theorem}\label{improvement of improvement of a theorem of AFG1}
Let $n\geq m$ be positive integers, $k:=n-m+1$ and let $f:\R^n\to\R^m$ be such that for every $x$ there exists a polynomial $P(x; \cdot)$ centered at $x$ such that
$$\limsup_{y\to x}\frac{|f(y)-P(x;y)|}{|y-x|^k}<\infty$$ for every $x\in\R^n$.
Then $\mathcal{L}^{m}\left(f(C_f) \right)=0$, where
$C_f$ is defined as the set $\{ x\in\R^n : \textrm{rank}\left(DP(x; \cdot)(x)\right) < m\}$.

\noindent The same statement holds true if $\R^n$ is replaced with an open subset of $\R^n$. 
\end{theorem}

\section*{Acknowledgement}
We wish to thank the referees for many suggestions that greatly improved this work, and especially for drawing our attention to the paper \cite{LiuTai}.


\begin{thebibliography}{}

\bibitem{AFG}
D. Azagra, J. Ferrera, and J. G\'omez-Gil, {\em The Morse-Sard theorem revisited}, preprint, arXiv:1511.05822.

\bibitem{AbrahamRobbin}
R. Abraham, J. Robbin, {\em Transversal mappings and flows}. W. A. Benjamin, Inc., New York--Amsterdam, 1967.

\bibitem{Alberti}
G. Alberti, {\em Generalized N-property and Sard theorem for Sobolev maps}, Atti Accad. Naz. Lincei Cl. Sci. Fis. Mat. Natur. Rend. Lincei (9) Mat. Appl. 23 (2012), no. 4, 477--491.

\bibitem{BarbetDambrineDaniilidis}
L. Barbet, M. Dambrine, and A. Daniilidis, {\em The Morse-Sard theorem for Clarke critical values}. Adv. Math. 242 (2013), 217--227.

\bibitem{Bates}
S. M. Bates, {\em Toward a precise smoothness hypothesis in Sard's theorem}, Proc. Amer. Math. Soc. 117 (1993), no. 1, 279--283.

\bibitem{BoHaSt}
B. Bojarski, P. Haj\l asz and P. Strzelecki, {\em Sard's theorem for mappings in H\"older and Sobolev spaces}, Manuscripta Math. 118 (2005), 383--397.

\bibitem{BolteDaniilidisLewis}
J. Bolte, A. Daniilidis, and A. Lewis, {\em A nonsmooth Morse-Sard theorem for subanalytic functions.} J. Math. Anal. Appl. 321 (2006), no. 2, 729--740.

\bibitem{BourKoKris1}
J. Bourgain, M. V. Korobkov and J. Kristensen, {\em On the Morse-Sard property and level sets of Sobolev and BV functions}, Rev. Mat. Iberoam. 29 (2013), no. 1, 1--23.

\bibitem{BourKoKris2}
J. Bourgain, M. V. Korobkov and J. Kristensen, {\em On the Morse-Sard property and level sets of $W^{n,1}$ Sobolev functions on $\R^n$},
J. Reine Angew. Math. 700 (2015), 93--112.

\bibitem{Campanato}
S. Campanato, {\em Propriet\`a di una famiglia di spazi funzionali},
Ann. Scuola Norm. Sup. Pisa (3) 18 1964 137--160. 

\bibitem{ClarkeEtAl}
F.H. Clarke, Yu.S. Ledyaev, R.J. Stern, and P.R. Wolenski, {\em Nonsmooth Analysis and Control Theory.} Grad. Texts in Math. 178, Springer, 1998.

\bibitem{CIL}
M.G. Crandall, H. Ishii, P.-L. Lions, {\em User's guide to viscosity solutions of second order partial differential equations}, Bull. Amer. Math. Soc. 27 (1992) 1--67.

\bibitem{DePascale}
L. De Pascale, {\em The Morse-Sard theorem in Sobolev spaces}, Indiana Univ. Math. J. 50 (2001), 1371--1386.

\bibitem{Dorronsoro}
J.R. Dorronsoro, {\em Differentiability properties of functions with bounded variation}, Indiana Univ. Math. J. 38 (1989), no. 4, 1027--1045. 

\bibitem{Dubo}
A. Y. Dubovitski\v{\i}, {\em Structure of level sets for differentiable mappings of an n-dimensional cube into a k-dimensional cube} (Russian), Izv. Akad. Nauk SSSR Ser. Mat. 21 (1957), no. 3, 371--408.

\bibitem{EvansGariepy}
L.C. Evans, R.F. Gariepy, {\em Measure theory and fine properties of functions}. Studies in Advanced Mathematics. CRC Press, Boca Raton, FL, 1992.

\bibitem{Ferrera}
J. Ferrera, {\em An introduction to nonsmooth analysis}. Elsevier/Academic Press, Amsterdam, 2014

\bibitem{Figalli}
A. Figalli, {\em A simple proof of the Morse-Sard theorem in Sobolev spaces}. Proc. Amer. Math. Soc. 136 (2008), no. 10, 3675--3681. 

\bibitem{Glaeser}
G. Glaeser, {\em Etudes de quelques alg\`ebres tayloriennes}, J. d'Analyse 6 (1958), 1-124.

\bibitem{Hajlasz2}
P. Haj\l asz, {\em Whitney's example by way of Assouad's embedding}, Proc. Amer. Math. Soc. 131 (2003), no. 11, 3463--3467.

\bibitem{HajlaszZimmerman}
P. Haj\l asz, S. Zimmerman, {\em Dubovitskij-Sard theorem for Sobolev mappings}, preprint, arXiv:1506.00025.

\bibitem{KorobkovKristensen}
M.V. Korobkov, J. Kristensen, {\em On the Morse-Sard theorem for the sharp case of Sobolev mappings}, Indiana Univ. Math. J. 63 (2014), no. 6, 1703--1724.

\bibitem{Landis}
E.M. Landis, {\em On functions representable as the difference of two convex functions}, Doklady Akad. Nauk SSSR (N.S.) 80 (1951), 9--11.

\bibitem{LinLiu}
C.-L. Lin, F.-C. Liu, {\em Approximate differentiability according to Stepanoff-Whitney-Federer}, Indiana Univ. Math. J. 62 (2013), no. 3, 855--868.

\bibitem{LiuTai}
F.-C. Liu, W.-S. Tai, {\em Approximate Taylor polynomials and differentiation of functions},  Topol. Methods Nonlinear Anal. 3 (1994), no. 1, 189--196.

\bibitem{Malgrange}
B. Malgrange, {\em Ideals of differentiable functions.} Tata Institute of Fundamental Research Studies in Mathematics, No. 3. Oxford University Press, 1967.

\bibitem{Morse}
A. P. Morse, {\em The behavior of a function on its critical set}, Ann. of Math. 40 (1939), 62--70.

\bibitem{Norton}
A. Norton, {\em A critical set with nonnull image has large Hausdorff dimension}, Trans. Amer. Math. Soc. 296 (1986), 367--376.

\bibitem{Norton2}
A. Norton, {\em Functions not constant on fractal quasi-arcs of critical points}, Proc. Amer. Math. Soc. 106 (1989) no. 2, 397--405.

\bibitem{NortonZMS}
A. Norton, {\em The Zygmund Morse-Sard theorem},
J. Geom. Anal. 4 (1994), no. 3, 403--424. 

\bibitem{PavZaj}
D. Pavlica and L. Zaj\'i\v{c}ek, {\em Morse-Sard theorem for d.c. functions and mappings on $\R^2$}, Indiana Univ. Math. J. 55 (2006), no. 3, 1195--1207.

\bibitem{Putten}
R. van der Putten, {\em The Morse-Sard theorem in $W^{n,n}(\Omega)$: a simple proof}, Bull. Sci. Math. 136 (2012), no. 5, 477--483.

\bibitem{Rifford}
L. Rifford, {\em On Viscosity Solutions of Certain Hamilton-Jacobi
Equations: Regularity Results and Generalized Sard's Theorems}
Comm. Partial Diff. Eq. 33 (2008), 517--559.

\bibitem{Sard}
A. Sard, {\em The measure of the critical values of differentiable maps}, Bull. Amer. Math. Soc. 48 (1942), 883--890.

\bibitem{Stein}
E. Stein, {\em Singular integrals and differentiability properties of functions}. Princeton, University Press, 1970.

\bibitem{Whitney1934}
H. Whitney, {\em Analytic extensions of differentiable functions defined in closed sets}, Trans. Amer. Math. Soc. 36 (1934), 63--89.

\bibitem{Whitney}
H. Whitney, {\em A function not constant on a connected set of critical points}, Duke Math. J. 1 (1935), 514--517.

\bibitem{Yomdin}
Y. Yomdin, {\em The geometry of critical and near-critical values of differentiable mappings}. Math. Ann. 264 (1983), no. 4, 495--515.

\end{thebibliography}
\end{document}